\documentclass[12pt,reqno]{amsart}
\usepackage{amssymb, a4wide, amsmath, mathtools,xy}
\xyoption{all}
\usepackage{latexsym,pdfsync,xcolor,graphicx}
\usepackage{multirow}

\usepackage[T1]{fontenc}

\usepackage[utf8]{inputenc}

\usepackage{array}
\usepackage{upgreek}
\usepackage{pstricks-add}
\usepackage{enumerate}
\usepackage{orcidlink}

\usepackage[english]{babel}	
\usepackage{comment}
\allowdisplaybreaks

\usepackage{pgf,tikz}
\usepackage{wrapfig}
\usetikzlibrary{arrows}

\usepackage{float}
\usepackage{appendix}

\usepackage{hyperref}
\usepackage{rotating}

\theoremstyle{plain}%
\newtheorem{theorem}{Theorem}[section]% meant for sectionwise numbers
\newtheorem{proposition}[theorem]{Proposition}%

\theoremstyle{definition}
\newtheorem{example}[theorem]{Example}%

\theoremstyle{remark}
\newtheorem{remark}[theorem]{Remark}
\newcommand{\E}{\mathcal{E}}

\DeclareMathOperator{\ann}{ann}
\DeclareMathOperator{\spa}{\textup{\textsf{span}}}

\DeclareMathOperator{\rank}{rank}
\DeclareMathOperator{\supp}{supp}

\numberwithin{equation}{section}
\title[Further results on modularity in evolution algebras]{Further results on modularity in evolution algebras}

\author[M. Ladra]{Manuel Ladra\textsuperscript{1}\,\orcidlink{0000-0002-0543-4508}}
\address{\textsuperscript{1}Department of Mathematics \& CITMAga, Universidade de Santiago de Compostela, 15782 Santiago de Compostela, Spain}
\email{manuel.ladra@usc.es, andresperez.rodriguez@usc.es}

\author[A. Pérez-Rodríguez]{Andrés Pérez-Rodríguez\textsuperscript{1}\,\orcidlink{0009-0007-1095-5328}}


\begin{document}
	
	\begin{abstract}
	In this paper, we study modularity in the context of evolution algebras. Although this property has been previously considered, a complete description is still missing in several natural settings. In particular, we obtain a full characterisation of modular evolution algebras in the nilpotent case and in the class of supersolvable regular evolution algebras.
	\end{abstract}

\maketitle

\section{Introduction}

Evolution algebras are commutative but nonassociative algebras introduced in 2006 by J.~P.~Tian and P.~Vojt\v{e}chovsk\'y \cite{TV_06}, motivated by models from non-Mendelian genetics. Their multiplication reflects the self-reproduction mechanism of genetic inheritance in asexual reproduction. A systematic study of these algebras was later developed in the monograph by Tian \cite{Tian_08}, which established the foundations of the theory and stimulated further research on their structural properties. A distinctive feature of evolution algebras is that, unlike most classical nonassociative algebraic varieties, they are not defined by polynomial identities but rather by the existence of a distinguished basis in which the product of two distinct basis elements is zero. Consequently, their study requires different tools and approaches. Despite these difficulties, evolution algebras have been widely investigated, including the study of their ideals, simplicity and semisimplicity, and their connections with graph theory (see, for instance, \cite{BCS_22,CKM_19,CMMT_25,EL_15}). 

One natural way to study the internal structure of an evolution algebra is through its subalgebra lattice. Actually, the relationship between an algebraic structure and the lattice formed by its substructures has long been an important theme in algebra. Early developments can be traced back to the work of Richard Dedekind, who first introduced the modular law (also called the Dedekind law). A lattice $(\mathcal{L},\vee,\wedge)$ is said to be \textit{modular} if, for every $x,y,z\in\mathcal{L}$ with $x\leq z$, one has
\begin{equation}\label{eq:mod}
	x\vee (y\wedge z)= (x\vee y)\wedge z.
\end{equation}
Equivalently, a lattice is modular if its Hasse diagram does not contain the pentagon lattice $N_5$ as a sublattice (see \cite[Theorem~2.1.2]{Gr_98}).

Modularity has been widely studied in several algebraic contexts. For instance, groups whose subgroup lattice is modular are known as $M$-groups. Moreover, the lattice of normal subgroups of any group is modular, and consequently, the subgroup lattice of an abelian group is modular as well. A similar phenomenon appears in module theory, where the lattice of submodules of a module over a ring is always modular. Related questions have also been considered in nonassociative settings such as Lie algebras (see \cite{Ge_76,Ko_65}), Leibniz algebras (see \cite{ST_22}), and restricted Lie algebras (see \cite{MPS_21,PST_23}).

Motivated by these developments, the study of modularity in the context of evolution algebras was initiated in \cite{LPP_25}. In that work, a necessary condition for modularity in the nilpotent case was obtained in terms of the existence of an absolute nilpotent element (see \cite[Proposition~5.4]{LPP_25}). Moreover, modularity was completely characterised in two extreme solvable cases: evolution algebras with one-dimensional derived subalgebra (see \cite[Lemma~6.1]{LPP_25}) and evolution algebras with maximum index of solvability (see \cite[Corollary~6.11]{LPP_25}). The aim of the present paper is to continue this line of research. After this introduction and the preliminary Section~\ref{sec:2}, we study modularity in the nilpotent case over $\mathbb{C}$ in Section~\ref{sec:3}, showing a close connection between this property and completeness. Finally, Section~\ref{sec:4} provides a complete characterisation of modularity for a particular class of regular evolution algebras, namely the supersolvable ones.

\section{Preliminaries}\label{sec:2}
An \textit{evolution algebra} over $\mathbb{K}$ is a $\mathbb{K}$-algebra $\mathcal{E}$ that admits a basis $B=\{e_1,\dots,e_n,\dots\}$, called a \textit{natural basis}, such that $e_ie_j=0$ for all $i\neq j$. In this note, we focus on finite-dimensional evolution algebras, meaning that $B$ is a finite set. For a given natural basis $B=\{e_1,\dots,e_n\}$ in $\mathcal{E}$, the scalars $a_{ij}\in\mathbb{K}$ satisfying $e_i^2=\sum_{j=1}^na_{ij}e_j$ are called the \textit{structure constants} of $\mathcal{E}$ relative to $B$. The matrix $M_B(\mathcal{E})=(a_{ij})_{i,j=1}^n$ is said to be the \textit{structure matrix} of $\mathcal{E}$ relative to $B$. Moreover, recall that the \textit{annihilator} of an evolution algebra $\mathcal{E}$ is characterised by \cite[Proposition~1.5.3]{thesis_yolanda},
\[\ann(\mathcal{E})\coloneqq\{u\in\mathcal{E}\colon u\mathcal{E}=0\}=\spa\{e_i\in B\colon e_i^2=0\}.\]

Given an evolution algebra $\mathcal{E}$ with a natural basis $B=\{e_1,\dots,e_n\}$ and an element $u=\sum_{i=1}^n\mu_ie_i\in\mathcal{E}$, we define its \textit{support} relative to $B$ as $\supp(u)\coloneqq\{i:\mu_i\neq0\}$. It is also known that evolution algebras are closed under quotients by ideals (see \cite[Lemma~1.4.11]{thesis_yolanda}). Let $\E$ be an evolution algebra with natural basis $B=\{e_1,\dots,e_n\}$ and $I$ an ideal of $\E$. Then, although $$B_{\E/I}\coloneqq\{\overline{e}=e+I\colon e\in B,\ e\notin I\}$$
is not necessarily a natural basis of $\E/I$, it always contains a natural basis of the quotient algebra (see \cite[Remark 1.4.12]{thesis_yolanda}). We recall that an ideal $I$ of $\E$ is called a \emph{basic ideal} if it admits a natural basis
consisting of vectors from the natural basis. Particularly, if we have a basic ideal $I=\spa\{e_i\colon i\in\varLambda\}$, with $\varLambda\subset\{1,\dots,n\}$, we will consider the set $B_{\E/I}=\{\overline{e_i}\colon i\notin\varLambda\}$ as the natural basis of the quotient algebra $\E/I$.

In this paper, we will mainly work with two opposite types of evolution algebras: nilpotent and regular. An evolution algebra is said to be \textit{nilpotent} if it admits a strictly triangular (upper or lower) structure matrix (see~\cite[Theorem~2.7]{CLOR_14}).  Hence, throughout this paper we assume that the structure matrix of a nilpotent evolution algebra $\E$ is 
\begin{equation}\label{eq:matrix}
	M_B(\E)=
	\left(
	\begin{matrix}
		0 & a_{12} & a_{13} & \dots & a_{1(k+1)} & \dots & a_{1n} \\
		0 & 0 & a_{23} & \dots & a_{2(k+1)} & \dots & a_{2n} \\
		\vdots & \vdots & \vdots & \ddots & \vdots & \ddots & \vdots \\
		0 & 0 & 0 & \dots & a_{k(k+1)} & \dots & a_{kn} \\
		0 & 0 & 0 & \dots & 0 & \dots & 0 \\
		\vdots & \vdots & \vdots & \ddots & \vdots & \ddots & \vdots \\
		0 & 0 & 0 & \dots & 0 & \dots & 0
	\end{matrix}
	\right),
\end{equation}
where the first $k$ rows are nonzero.

 At the other extreme are regular evolution algebras. An evolution algebra $\E$ is called \textit{regular} if $\E=\E^2$, or equivalently, if the matrix $M_B(\E)$ is nonsingular.
Several results concerning subalgebras of regular evolution algebras will be used throughout the paper. In particular, as proved in \cite[Lemma~1]{LP_25_regular}, one-dimensional subalgebras correspond to the nontrivial solutions of the system
\begin{equation*}\label{sist}
	\begin{pmatrix}
		x_1^2 \\ \vdots \\ x_n^2
	\end{pmatrix}
	=
	\big(M_B(\E)^t\big)^{-1}
	\begin{pmatrix}
		x_1 \\ \vdots \\ x_n
	\end{pmatrix}.
\end{equation*}
 Moreover, although evolution algebras are not closed under taking subalgebras in general, this property holds in the regular case; that is, every subalgebra of a regular evolution algebra admits a natural basis (see~\cite[Theorem~2]{LP_25_regular}). Finally, as stated in \cite[Corollary~7]{LP_25_regular}, if $\E$ is a three-dimensional regular evolution algebra over an algebraically closed field, a subspace $\spa\{e_i,v\}$ with $v\in\spa\{e_p,e_q\}$ and $i\neq p,q$ is a subalgebra if and only if
\[
a_{ip}^2a_{iq}a_{pp}+a_{iq}^3a_{qp}
=
a_{ip}^3a_{pq}+a_{ip}a_{iq}^2a_{qq}.
\]

For the reader's convenience, we conclude this preliminary section by recalling some lattice-theoretical notions that will appear in our study in the framework of evolution algebras. Let $\mathcal{E}$ be an evolution algebra. Following \eqref{eq:mod}, we say that $\mathcal{E}$ is \textit{modular} if $
\langle U,V\cap W\rangle=\langle U,V\rangle\cap W$
for all subalgebras $U,V,W$ of $\mathcal{E}$ with $U\subseteq W$. Note that modularity is preserved under subalgebras  and quotients. To study modularity we will use the following result, already essential in \cite{LPP_25}.

\begin{proposition}[{\cite[Theorem 2.2]{An_94}}]\label{prop:equiv_mod_qi}
Let $\mathcal{E}$ be a finite-dimensional solvable (not necessarily evolution) algebra. Then, $\E$ is modular if and only if every subalgebra of $\mathcal{E}$ is a quasi-ideal.
\end{proposition}

Recall that a subalgebra $U$ of $\mathcal{E}$ is called a \textit{quasi-ideal} of $\mathcal{E}$ if 
$\langle U,V\rangle=U+V$
for every subalgebra $V$ of $\mathcal{E}$. Clearly, every ideal is a quasi-ideal.

\begin{remark}\label{rem:flag}
The proof of Proposition~\ref{prop:equiv_mod_qi} relies on the existence of a \textit{complete flag of subalgebras}, that is, a chain of subalgebras
\[
0=U_0\subsetneq U_1\subsetneq\cdots\subsetneq U_n=\mathcal{E},
\]
such that $\dim U_i=i$ for all $0\le i\le n$. Consequently, the equivalence in Proposition~\ref{prop:equiv_mod_qi} remains valid for any finite-dimensional algebra admitting such a complete flag of subalgebras.
\end{remark}

\section{Completeness as a sufficient condition for modularity}\label{sec:3}
As already mentioned before, evolution algebras are not, in general, closed under taking subalgebras. Consequently, besides ordinary subalgebras and subalgebras admitting a natural basis, it is natural to consider subalgebras that admit a natural basis which can be extended to a natural basis of the whole algebra. This latter notion corresponds to the original concept of evolution subalgebra introduced by Tian in his monograph \cite[Definition~4, p.~23]{Tian_08}. However, in this article, we adopt the terminology from \cite[Definition~1.4.3]{thesis_yolanda}, as it is now the most commonly used in the literature. Accordingly, a subalgebra is called an \textit{evolution subalgebra} if it admits a natural basis, and we say that it has the \textit{extension property} if there exists a natural basis of the subalgebra that can be extended to a natural basis of the whole algebra.

As introduced in \cite{CKO_19}, an evolution algebra $\E$ is said to be \textit{complete} if every subalgebra of $\E$ is an evolution subalgebra with the extension property. Moreover, as shown in \cite[Proposition~2, p.~24]{Tian_08}, any such subalgebra is an ideal, and therefore a quasi-ideal. This observation suggests a close connection between completeness and modularity. 
\begin{proposition}\label{prop:comp_imp_mod}
	Let $\E$ be a complex evolution algebra. If $\E$ is complete, then $\E$ is modular.
\end{proposition}
\begin{proof}
If $\E$ is complete, then, using the classification theorems first established in part and further conjectured in \cite[Theorem~4.2 and Conjectures~5.2 \&~5.3]{CKO_19} and later fully proved in \cite[Corollary~3.4]{GP_25_complete}, it is isomorphic to one of the following pairwise nonisomorphic algebras:
	\[\{e_1^2=e_1\},\qquad\widetilde{\E}\oplus\mathbb{C}^{n-k},\qquad\{e_1^2=e_1\}\oplus \mathbb{C}^{n-1}\qquad\text{or}\qquad\{e_1^2=e_1\}\oplus \widetilde{\E}\oplus\mathbb{C}^{n-k-1},\]
	where $\widetilde{\E}$ is a $k$-dimensional evolution algebra with maximal index of nilpotency and $\mathbb{C}^s$ denotes the $s$-dimensional zero evolution algebra over $\mathbb{C}$. It is easy to check that all these algebras admit a complete flag of subalgebras. Then, by Proposition~\ref{prop:equiv_mod_qi} and Remark~\ref{rem:flag}, the result follows.
\end{proof}

Although completeness is a sufficient condition for modularity, it is not necessary in general.
\begin{example}[{\cite[Example~5.3]{LPP_25}}]
Let $\E$ be the evolution algebra over $\mathbb{C}$ with natural basis $\{e_1,e_2,e_3\}$ and multiplication given by $e_1^2=-e_2^2=e_1+e_2$ and $e_3^2=e_2$. The nonzero proper subalgebras of $\E$ are $\spa\{e_1+e_2\}$, $\spa\{e_1-e_2\}$ and $\spa\{e_1,e_2\}$. Note that $\E$ is solvable and all the above subalgebras are quasi-ideals, so $\E$ is modular. However, it is easy to check that $e_1+e_2$ cannot be extended to a natural basis of the whole algebra (see~\cite[Theorem~2.4]{BCS_22}).
\end{example}
Our aim now is to prove that completeness and modularity are equivalent in the context of nilpotent evolution algebras over $\mathbb{C}$. 
 To this end, we recall \cite[Proposition~5.4]{LPP_25}, which states that if $\E$ is a nilpotent evolution algebra over a field of characteristic not $2$ and $\E$ is modular, then there is no absolute nilpotent element outside the annihilator. Over $\mathbb{C}$ (or, more generally, over a quadratically closed field of characteristic not $2$), this result can be equivalently reformulated as follows: if $\E$ is a modular evolution algebra, then, assuming without loss of generality that its structure matrix is~\eqref{eq:matrix}, we have that $\rank\big(M_B(\E)\big)=k$. Indeed, $\rank\big(M_B(\E)\big)<k$ if and only if there exists a nonzero linear relation
 \[
 \alpha_1 e_1^2 + \cdots + \alpha_k e_k^2 = 0,
 \]
 which yields the absolute nilpotent element $\sqrt{\alpha_1}e_1 + \cdots + \sqrt{\alpha_k}e_k$.

\begin{proposition}\label{prop:nec_mod}
Let $\E$ be a nilpotent evolution algebra over $\mathbb{C}$ (or, more generally, over a quadratically closed field of characteristic not $2$). If $\E$ is modular, then there is no element $u\in\E$ such that:
\begin{itemize}
	\item $|\supp(u)|\geq 2$,
	\item $\supp(u)\cap\supp(\ann(\E))=\emptyset$, and
	\item $u^2\in\ann(\E)$.
\end{itemize}
\end{proposition}
\begin{proof}
	Consider a natural basis $B=\{e_1,\dots,e_n\}$ such that the structure matrix satisfies~\eqref{eq:matrix}. Suppose, for contradiction, that there exists an element $u\in\E$ satisfying the above conditions. Then  $u=\alpha_{1}e_{1}+\dots+\alpha_{l}e_{l}$, where $\alpha_l\neq0$, $l\leq k$, $|\supp_B(u)|\geq2$, and $u^2\in\ann(\E)$. 
	Observe that, by \cite[Proposition~5.4]{LPP_25}, we necessarily have that $\rank\big(M_B(\mathcal{E})\big)=k$. Consequently,
	$0\neq u^2\in\ann(\E)$.  We distinguish the two following cases:
	\begin{enumerate}
		\item There exist at least two indices $1\leq i,j\leq k$ such that $e_i^2,e_j^2\in\ann(\E)$. Consider the elements $u=e_i+e_j$ and $v=e_i-e_j$. Since $\rank(M_B(\E))=k$, we have that $0\neq u^2=v^2\in\ann(\E)$. Then, it suffices to check that the subalgebras  $\spa\{u,u^2\}$ and $\spa\{v,u^2\}$ are not quasi-ideals. Indeed,
		\[\langle u,v,u^2\rangle=\langle e_i,e_j,u^2\rangle\neq\spa\{e_i,e_j,u^2\},\]
		which follows from the fact that $e_i^2$ and $e_j^2$ are linearly independent elements of $\ann(\E)$, and therefore both cannot be scalar multiples of $u^2$.
		\item The only index satisfying $e_k^2\in\ann(\E)$ is $k$. In this case, we may assume that the element $u$ is of the form $\alpha_1e_1+\dots+\alpha_le_l$ with $l<k$. Consider the quotient algebra $\E/\ann(E)$, with natural basis $B_{\E/\ann(\E)}=\{\overline{e_1},\dots,\overline{e_k}\}$. Then, $u$ becomes an absolute nilpotent element in $\E/\ann(\E)$ outside $\ann\big(\E/\ann(\E)\big)=\spa\{\overline{e_k}\}$, and therefore this quotient algebra is not modular by \cite[Proposition~5.4]{LPP_25}. Since modularity is preserved under quotients, we obtain a contradiction.
	\end{enumerate}
	In either case, we obtain that 
	$\E$ is not modular. Therefore, the result follows.
\end{proof}

\begin{remark}
The previous result may fail over fields which are not quadratically closed. 
For instance, consider the evolution algebra $\E$ over $\mathbb{R}$ with natural basis $\{e_1,e_2,e_3\}$ and multiplication given by $e_1^2=e_2^2=e_3$ and $e_3^2=0$ (see \cite[Example~5.6]{LPP_25}). In this case, the element $u=e_1+e_2$ satisfies the conditions of the previous result. Nevertheless, the algebra $\E$ turns out to be modular, as shown by its lattice of subalgebras below.
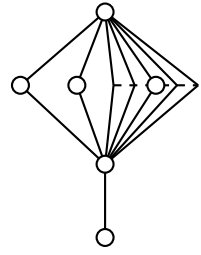
\begin{figure}
	\begin{center}
		\begin{minipage}{0.65\textwidth}
		%\hspace{-0.2cm}
		\begin{center}
			%\vspace{-0.5cm}
			\begin{tabular}{| l | l |}
				\hline
				\textbf{Subalg. of dim. $1$} & \textbf{Subalg. of dim. $2$}  \\
				\hline
				$\spa\{e_3\}$  & $\spa\{e_1,e_3\}$ \\
				& $\spa\{e_2,e_3\}$  \\
				& $\spa\{e_1+\alpha e_2,e_3\},\alpha\in\mathbb{R}^*$\\
				\hline
			\end{tabular}
		\end{center}
	\end{minipage}
	\begin{minipage}{0.25\textwidth}
		\begin{center}%\hspace{-0.6cm}
			\begin{tikzpicture}[scale=0.75]
				\draw[dashed,thick]  (4.65,0.7) -- (6.15,0.7);
				\draw[thick] (4.5,-0.7) -- (4.5,-2);
				\draw[thick] (4.5,2) -- (3,0.7);
				\draw[thick] (4.5,2) -- (4,0.7);
				\draw[thick] (4.5,2) -- (5.4,0.7);
				\draw[thick] (4.5,2) -- (4.65,0.7);
				\draw[thick] (4.5,2) -- (6.15,0.7);
				\draw[thick] (4.5,2) -- (5.025,0.7);
				\draw[thick] (4.5,2) -- (5.775,0.7);
				\draw[thick] (4.5,-0.7) -- (3,0.7);
				\draw[thick] (4.5,-0.7) -- (4,0.7);
				\draw[thick] (4.5,-0.7) -- (5.4,0.7);
				\draw[thick] (4.5,-0.7) -- (4.65,0.7);
				\draw[thick] (4.5,-0.7) -- (6.15,0.7);
				\draw[thick] (4.5,-0.7) -- (5.025,0.7);
				\draw[thick] (4.5,-0.7) -- (5.775,0.7);

				\draw[fill=white,thick=black] (4.5,-2) circle [radius=0.15];
				\draw[fill=white,thick=black](4.5,-0.7) circle [radius=0.15];
				\draw[fill=white,thick=black](5.4,0.7) circle [radius=0.15];
				\draw[fill=white,thick=black](3,0.7) circle [radius=0.15];
				\draw[fill=white,thick=black](4,0.7) circle [radius=0.15];
				\draw[fill=white,thick=black](4.5,2) circle [radius=0.15];			
			\end{tikzpicture}
		\end{center}
	\end{minipage}
	\end{center}
	\caption{Proper subalgebras and subalgebra lattice of the real evolution algebra with product $e_1^2=e_2^2=e_3$ and $e_3^2=0$.}
\end{figure}
	
\end{remark}

\begin{theorem}\label{th:equiv_mod}
	Let $\E$ be a nilpotent evolution algebra of dimension $n$ over $\mathbb{C}$. Then, the following statements are equivalent:
	\begin{enumerate}[\rm (i)]
		\item $\E$ is modular;
		\item $\E$ is isomorphic to $\widetilde{\E}\oplus\mathbb{C}^{n-k}$, where $\widetilde{\E}$ is a $k$-dimensional evolution algebra with maximum index of nilpotency, and $\mathbb{C}^{n-k}$ denotes the $(n-k)$-dimensional zero evolution algebra over $\mathbb{C}$; and
		\item $\E$ is complete.
	\end{enumerate}
\end{theorem}
\begin{proof}
 	The equivalence (ii)$\Leftrightarrow$(iii) was established in \cite[Theorem~4.2]{CKO_19}, and the implication (iii)$\Rightarrow$(i) was proved in Proposition~\ref{prop:comp_imp_mod}. It therefore suffices to prove the implication (i)$\Rightarrow$(ii).
	
	Let $\E$ be a modular nilpotent evolution algebra whose multiplication, without loss of generality, satisfies~\eqref{eq:matrix}.
	Since modularity is preserved under quotients, the quotient algebra $\E/\ann(\E)$ is also modular. Moreover, as a straightforward consequence of Proposition~\ref{prop:nec_mod}, this quotient algebra has a one-dimensional annihilator. Indeed, if there existed two elements $\overline{e_i},\overline{e_j}\in\ann(\E/\ann(\E))$, then  $e_i^2,e_j^2\in\ann(\E)$, and hence the element $u=e_i+e_j$ would satisfy the conditions of Proposition~\ref{prop:nec_mod}. 
	
	In addition, by
	\cite[Theorem~5.1 \& Corollary~5.7]{LPP_25}, it follows that $\E/\ann(\E)$ has maximum index of nilpotency.  Therefore, the structure constants $a_{12},a_{23},\dots,a_{(k-1)k}$ are all nonzero. 
	Consequently, following the argument used in \cite[p.~294]{CKO_19}, we consider a change of basis of the form \[e_i'=e_i+\sum_{j=k+1}^{n}\beta_{ij}e_j,\quad1\leq i\leq k,\qquad\quad e_i'=e_i,\quad k+1\leq i\leq n,\] where the coefficients $\beta_{ij}$ are determined by the system
	\begin{align*}
		\left(
		\begin{matrix}
			0 & a_{12} & a_{13} & \dots & a_{1k} \\
			0 & 0 & a_{23} & \dots & a_{2k} \\
			\vdots & \vdots & \vdots & \ddots & \vdots \\
			0 & 0 & 0 & \dots & a_{(k-1)k}
		\end{matrix}
		\right)&
		\left(
		\begin{matrix}
			\beta_{1(k+1)} & \beta_{1(k+2)} & \dots & \beta_{1n} \\
			\beta_{2(k+1)} & \beta_{2(k+2)} & \dots & \beta_{2n} \\
			\vdots & \vdots & \ddots & \vdots \\
			\beta_{k(k+1)} & \beta_{k(k+2)} & \dots & \beta_{kn}
		\end{matrix}
		\right)
		\\&=
		\left(
		\begin{matrix}
			a_{1(k+1)} & a_{1(k+2)} & \dots & a_{1n} \\
			a_{2(k+1)} & a_{2(k+2)} & \dots & a_{2n} \\
			\vdots & \vdots & \ddots & \vdots \\
			a_{(k-1)(k+1)} & a_{(k-1)(k+2)} & \dots & a_{(k-1)n}
		\end{matrix}
		\right).
	\end{align*}
	With this change of basis, we obtain that $\E$ is isomorphic to the direct sum $\widetilde{\E}\oplus\mathbb{C}^{n-k}$, where $\{e_1',\dots,e_k'\}$ is a natural basis for $\widetilde{\E}$.
\end{proof}
\begin{remark}
Although stated over $\mathbb{C}$, the classification of complete nilpotent evolution algebras given in \cite[Theorem~4.2]{CKO_19} remains valid over quadratically closed fields $\mathbb{K}$ of characteristic not two. In this more general setting, such algebras are still isomorphic to $\widetilde{\E}\oplus\mathbb{K}^{n-k}$. Consequently, Theorem~\ref{th:equiv_mod} extends to quadratically closed fields of characteristic different from $2$.
\end{remark}
\section{Modularity in a specific case of regular evolution algebras}\label{sec:4}
Recall that an (evolution) algebra $\E$ is said to be \textit{supersolvable} if there exists a complete flag made up of ideals, that is, there exist a chain 
\[0=I_0\subsetneq I_1\subsetneq\dots\subsetneq I_n=\E\] of ideals such that $\dim{I_i}=i$ for all $0\leq i\leq n$. As stated in \cite[Remark~4.2]{LPP_25}, and in contrast to the case of groups or Lie algebras, there exist supersolvable evolution algebras that are not solvable. For instance, consider the regular evolution algebra with natural basis $\{e_1,\dots,e_n\}$ and product $e_i^2=e_i$ for all $i=1,\dots,n$. Indeed, 
\[0\subsetneq\spa\{e_1\}\subsetneq\spa\{e_1,e_2\}\subsetneq\dots\subsetneq\spa\{e_1,\dots,e_{n-1}\}\subsetneq\E\] is a complete flag made up of ideals.  We now state the following elementary result, which characterises supersolvability in the regular setting.
\begin{proposition}\label{prop:super_triangular}
Let $\mathcal{E}$ be a regular evolution algebra over any field $\mathbb{K}$. Then, $\mathcal{E}$ is supersolvable if and only if there exists a natural basis such that the structure matrix is lower triangular with ones on the diagonal.
\end{proposition}
\begin{proof}
The sufficiency is immediate: such a structure matrix directly yields a complete flag of ideals: $0\subsetneq\spa\{e_1\}\subsetneq\dots\subsetneq\spa\{e_1,\dots,e_{n-1}\}\subsetneq\E$.

For the necessity, note that in the regular setting, all ideals are basic (see \cite[Proposition~4.2]{BCS_22}), which means that all of them admit a basis consisting of vectors of the natural basis of $\E$. This implies that a triangular matrix can be obtained by simply reordering the elements of the natural basis; let $(a_{ij})$ denote such a structure matrix. Finally, to normalise the diagonal entries to ones, it suffices to consider the natural basis $\{f_i=\frac{1}{a_{ii}}e_i\}_{i=1}^n$.
\end{proof}

In particular, our purpose is to characterise modularity within three-dimensional supersolvable regular evolution algebras over $\mathbb{C}$ (and, in fact, to go beyond this case). Note that, by the previous proposition, whenever working with this type of evolution algebras, we may assume that their structure matrices are
\begin{equation}\label{eq:matrix_sup}
	\begin{pmatrix}
		1 & 0 & 0\\
		\lambda & 1 & 0\\
		\mu & \rho & 1
	\end{pmatrix},\text{ with }\lambda, \mu,\rho\in\mathbb{C}.
\end{equation}
Consequently, an evolution algebra over $\mathbb{C}$ which admits a natural basis such that the structure matrix is \eqref{eq:matrix_sup} will be denoted by $\mathcal{E}_\mathbb{C}^{\mathrm{reg}}(\lambda,\mu,\rho)$. Moreover, by Lemma~\cite[Lemma~1]{LP_25_regular}, their one-dimensional subalgebras are given by the nontrivial solutions of the following nonlinear system of equations:
\begin{align}\label{eq:sist_supersoluble}
	\left\{
		\begin{array}{ccccccc}
			x^2&=&x&-&\lambda y&+&(\lambda\rho-\mu)z,\\
			y^2&=&&&y&-&\rho z,\\
			z^2&=&&&&&z.\\
		\end{array}
		\right.
\end{align}
\begin{remark}\label{rem:qi_super}
	Observe that every supersolvable algebra trivially admits a complete flag of subalgebras. Hence, by Remark~\ref{rem:flag}, the equivalence stated in Proposition~\ref{prop:equiv_mod_qi} also applies to supersolvable algebras.
\end{remark}
\begin{proposition}\label{prop:cond_mod}
	If an evolution algebra 
	$\E=\mathcal{E}_\mathbb{C}^{\mathrm{reg}}(\lambda,\mu,\rho)$ is modular, then the following assertions hold:
\begin{enumerate}[\rm (i)]
\item $\rho\neq0$;
\item $\rho^2\lambda-\rho\mu+\mu^2=0$; and 
\item $\lambda=0$.
\end{enumerate}
\end{proposition}
\begin{proof}
Since we are working over an algebraically closed field, the system~\eqref{eq:sist_supersoluble} always admits a solution with all components nonzero. Consequently, a subalgebra of the form $\spa\{\alpha e_1+\beta e_2+e_3\}$, with $\alpha$ and $\beta$ nonzero, always exists. Denote by $U$ one of them. Similarly, it is easy to check that $\spa\{\gamma e_1+e_2\}$ is a subalgebra if and only if $\gamma$ is a solution of the quadratic equation $x^2-x+\lambda=0$. Since this equation has at least one nonzero solution, we denote by $V$ a subalgebra of the form $\spa\{\gamma e_1+e_2\}$ with $\gamma\neq0$.

(i) For the sake of contradiction, assume that $\rho=0$. In this case, observe that $\spa\{\delta e_1+e_3\}$ is a subalgebra if and only if $\delta$ is a solution of $x^2-x+\mu=0$. As before, this equation has at least one nonzero solution, so denote by $W$ a subalgebra of the form $\spa\{\delta e_1+e_3\}$ with $\delta\neq0$.  By \cite[Theorem~2]{LP_25_regular}, we then obtain that $V+W$ is not a subalgebra, since it does not admit a natural basis. Therefore, they are not quasi-ideals, and the claim follows from Remark~\ref{rem:qi_super}.

(ii) Since $\E$ is modular then, Remark~\ref{rem:qi_super} implies that $U+V$ is a subalgebra. Then, by \cite[Theorem~2]{LP_25_regular},  $U+V$  admits a natural basis, which occurs precisely when  $\alpha e_1+\beta e_2=k(\gamma e_1+e_2)$ for some $k\in\mathbb{C}^*$.
This leads to a two-dimensional
subalgebra $\spa\{v,e_3\}$, with $v\in\spa\{e_1,e_2\}$. According to \cite[Corollary~7]{LP_25_regular}, such a subalgebra exists if the identity 
\[\mu^2\rho+\rho^3\lambda=\rho^2\mu\implies\rho(\rho^2\lambda-\rho\mu+\mu^2)=0.\] 
is satisfied. Since $\rho\neq0$ by item (i), the claim is established.

(iii) Since $\rho\neq0$ by item (i), the only possible two-dimensional subalgebra of the form $\spa\{v,e_3\}$, with $v\in\spa\{e_1,e_2\}$, is $\spa\{\mu e_1+\rho e_2,e_3\}$. Moreover, consider the subalgebras $V_1=\spa\{\gamma_1e_1+e_2\}$ and $V_2=\spa\{\gamma_2e_1+e_2\}$ with $\gamma_1,\gamma_2$ solutions of $x^2-x+\lambda=0$. We now claim that if $\gamma_1$ and $\gamma_2$ are both nonzero, then $\gamma_1=\gamma_2$, that is, $x^2-x+\lambda=0$ has only one nonzero solution. Indeed, in this case, as all subalgebras are quasi-ideals, we have
\[\langle U, V_1\rangle=\langle U, V_2\rangle=\spa\{\mu e_1+\rho e_2,e_3\},\] 
which implies that $\mu e_1+\rho e_2=k_1(\gamma_1e_1+e_2)=k_2(\gamma_2e_1+e_2)$ with $k_1,k_2\in\mathbb{C}$, what yields that $\gamma_1=\gamma_2$. Consequently, we have that either $\lambda=0$ or $\lambda=\frac{1}{4}$.

We now show that the case $\lambda=\frac{1}{4}$ is not valid. For the sake of contradiction, if $\lambda=\frac{1}{4}$, then by item (ii), we have that $\mu=\frac{1}{2}\rho$. It is easy to check that, in this case, $\spa\{\frac{1}{2}e_1+e_2\}$ is a subalgebra. Moreover, looking at \eqref{eq:sist_supersoluble}, all one-dimensional evolution algebras of the form $\spa\{\alpha e_1+\beta e_2+e_3\}$, with $\alpha\neq0$ or $\beta\neq0$, are given by the nontrivial solutions of  
\begin{align}\label{eq:sist}
	\left\{
		\begin{array}{ccccccc}
			x^2&=&x&-&\frac{1}{4} y&-&\frac{1}{4}\rho,\\
			y^2&=&&&y&-&\rho.\\
		\end{array}
		\right.
\end{align}
In fact, all these subalgebras must satisfy
\[\left\langle\frac{1}{2}e_1+e_2,\alpha e_1+\beta e_2+e_3\right\rangle=\spa\left\{\frac{1}{2}e_1+e_2,\alpha e_1+\beta e_2+e_3\right\}.\]
Consequently, it must hold that $\beta=2\alpha$ for all $(\alpha,\beta)$ nontrivial solutions of \eqref{eq:sist}. Then, we get that $\alpha^2=\frac{1}{2}\alpha-\frac{1}{4}\rho$ and $4\alpha^2=\alpha-\rho$, where we easily obtain that $\alpha=\rho=0$, a contradiction with item (i). 
\end{proof}

\begin{theorem}\label{th:char_mod}
Let $\mathcal{E}$ be a three-dimensional supersolvable regular evolution algebra over $\mathbb{C}$. Then, $\mathcal{E}$ is modular if and only if it is isomorphic to $\mathcal{E}_\mathbb{C}^{\mathrm{reg}}\big(0,\frac{1}{4},\frac{1}{4}\big)$.
\end{theorem}
\begin{proof}
If $\E$ is modular, by Proposition~\ref{prop:cond_mod}, we have that $\lambda=0$, so $\spa\{e_1+e_2\}$ is a subalgebra. Moreover, any one-dimensional subalgebra $\spa\{\alpha e_1+\beta e_2+e_3\}$ is given by 
\begin{align}\label{eq:syst_lambda_0}
	\left\{
	\begin{array}{ccccccc}
		x^2&=&x&&&-&\mu,\\
		y^2&=&&&y&-&\rho.\\
	\end{array}
	\right.
\end{align}
Since $\mathcal{E}$ is modular and by \cite[Theorem~2]{LP_25_regular}, we necessarily have
$$\langle e_1+e_2,\alpha e_1+\beta e_2+e_3\rangle=\spa\{e_1+e_2,\alpha e_1+\beta e_2+e_3\}=\spa\{e_1+e_2,e_3\}$$ for any $\alpha$ and $\beta$ solutions of $x^2-x+\mu=0$ and $y^2-y+\rho=0$, respectively. Notice that this holds if and only if $\alpha=\beta$ in any case, i.e. if and only if $\rho=\mu=\frac{1}{4}$.

Conversely, by a straightforward computation, we show that the one-dimensional subalgebras of  $\mathcal{E}_\mathbb{C}^{\mathrm{reg}}\big(0,\frac{1}{4},\frac{1}{4}\big)$ are $\spa\{e_1\}$, $\spa\{e_2\}$, $\spa\{e_1+e_2\}$ and $\spa\{e_1+e_2+2e_3\}$, which are all quasi-ideals. Since every subalgebra of codimension one is a quasi-ideal, all its subalgebras are quasi-ideals, and consequently $\mathcal{E}_\mathbb{C}^{\mathrm{reg}}\big(0,\frac{1}{4},\frac{1}{4}\big)$ is modular.\\

\begin{figure}[H]
	\noindent\begin{minipage}{0.65\textwidth}
		\centering
			\begin{tabular}{|| l | l ||}
				\hline
				\textbf{Subalg. of dim. $1$} & \textbf{Subalg. of dim. $2$}  \\
				\hline\hline
				$\spa\{e_1\}$  & $\spa\{e_1,e_2\}$\\
				$\spa\{e_2\}$   & $\spa\{e_1,e_2+2e_3\}$  \\
				$\spa\{e_1+e_2\}$ & $\spa\{e_2,e_1+2e_3\}$  \\
				$\spa\{e_1+e_2+2e_3\}$ & $\spa\{e_3,e_1+e_2\}$ \\
				\hline
			\end{tabular}
	\end{minipage}
	\begin{minipage}{0.34\textwidth}
		\centering
			\begin{tikzpicture}[scale=0.8]
				\draw[thick] (0,-1.5) -- (-0.5,-0.5);
				\draw[thick] (0,-1.5) -- (-1.5,-0.5);
				\draw[thick] (0,-1.5) -- (1.5,-0.5);
				\draw[thick] (0,-1.5) -- (0.5,-0.5);
				
				\draw[thick] (0,1.5) -- (-0.5,0.5);
				\draw[thick] (0,1.5) -- (-1.5,0.5);
				\draw[thick] (0,1.5) -- (1.5,0.5);
				\draw[thick] (0,1.5) -- (0.5,0.5);						
				
				\draw[thick] (-1.5,-0.5) -- (-1.5,0.5);
				\draw[thick] (-0.5,-0.5) -- (-1.5,0.5);
				\draw[thick] (0.5,-0.5) -- (-1.5,0.5);
				
				\draw[thick] (-1.5,-0.5) -- (-0.5,0.5);
				\draw[thick] (1.5,-0.5) -- (-0.5,0.5);
				
				\draw[thick] (-0.5,-0.5) -- (0.5,0.5);
				\draw[thick] (1.5,-0.5) -- (0.5,0.5);
				
				\draw[thick] (1.5,-0.5) -- (1.5,0.5);
				\draw[thick] (0.5,-0.5) -- (1.5,0.5);

				\draw[fill=white,thick=black](0.5,-0.5) circle [radius=0.1125];
				\draw[fill=white,thick=black](1.5,-0.5) circle [radius=0.1125];
				\draw[fill=white,thick=black](-0.5,-0.5) circle [radius=0.1125];
				\draw[fill=white,thick=black](-1.5,-0.5) circle [radius=0.1125];

				\draw[fill=white,thick=black](0.5,0.5) circle [radius=0.1125];
				\draw[fill=white,thick=black](1.5,0.5) circle [radius=0.1125];
				\draw[fill=white,thick=black](-0.5,0.5) circle [radius=0.1125];
				\draw[fill=white,thick=black](-1.5,0.5) circle [radius=0.1125];

				\draw[fill=white,thick=black] (0,-1.5) circle [radius=0.1125];
				
				\draw[fill=white,thick=black] (0,1.5) circle [radius=0.1125];
			\end{tikzpicture}
	\end{minipage}
\caption{Proper subalgebras and subalgebra lattice of $\mathcal{E}_\mathbb{C}^{\mathrm{reg}}\big(0,\frac{1}{4},\frac{1}{4}\big)$.}
\label{example:reg_sup}
\end{figure}
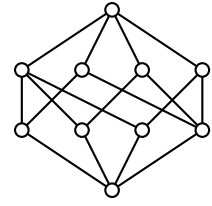
\end{proof}
Moreover, it has been possible to characterise modularity in supersolvable regular evolution algebras over $\mathbb{C}$ in arbitrary dimension, as shown in the next theorem.
\begin{theorem}\label{th:char_sup_reg}
Let $\mathcal{E}$ be a supersolvable regular evolution algebra over $\mathbb{C}$. Then, $\mathcal{E}$ is modular if and only if $n\leq2$ or $\mathcal{E}\cong\mathcal{E}_\mathbb{C}^{\mathrm{reg}}\big(0,\frac{1}{4},\frac{1}{4}\big)$.
\end{theorem}
\begin{proof}
Since every algebra of dimension at most two is modular, and the three-dimensional case has been characterised in Theorem~\ref{th:char_mod}, it suffices to consider the case $n>3$.
	
Let $\mathcal{E}$ be an $n$-dimensional supersolvable regular evolution algebra with natural basis $B=\{e_1,\dots,e_n\}$ and structure matrix $M_B(\mathcal{E})=(a_{ij})$ as in Proposition~\ref{prop:super_triangular}. If $\mathcal{E}$ is modular, then the subalgebra $\spa\{e_1,e_2,e_3\}$ would also be modular, since modularity is preserved under subalgebras. By Theorem~\ref{th:char_mod}, this forces that $a_{21}=0$ and $a_{31}=a_{32}=\frac{1}{4}$.
Now consider the quotient algebra $\mathcal{E}/\spa\{e_1\}$. Its structure matrix with respect to the natural basis $B_{\E/\spa\{e_1\}}=\{\overline{e_2},\dots,\overline{e_n}\}$ is obtained by deleting the first row and the first column of $M_B(\mathcal{E})$. In particular, the entry $(2,1)$ of this matrix equals $\tfrac14$. Consequently, Theorem~\ref{th:char_mod} implies that the three-dimensional subalgebra $\spa\{\overline{e_2},\overline{e_3},\overline{e_4}\}$ is not modular. Hence this quotient algebra is not modular.
Since modularity is preserved under both subalgebras and quotients, we obtain a contradiction. 

Therefore no supersolvable regular evolution algebra of dimension greater than three can be modular, and the result follows.
\end{proof}

\begin{remark}
To conclude, note that although formulated over $\mathbb{C}$, the statements and proofs of Proposition~\ref{prop:cond_mod}and Theorems~\ref{th:char_mod} and~\ref{th:char_sup_reg} remain valid over any algebraically closed field of characteristic different from $2$. 
\end{remark}

\section*{Acknowledgments}
This work was partially supported by the projects PID2020-115155GB-I00 and PID2024-155502NB-I00 granted by MICIU/AEI/10.13039/501100011033; 
and by Xunta de Galicia through the Competitive Reference Groups (GRC), ED431C
2023/31.
The second author was also supported by the predoctoral contract FPU21/05685 from the Ministerio de Ciencia, Innovación y Universidades (Spain).


\begin{thebibliography}{10}
	
	\bibitem{An_94}
	J.~A. Anquela, \emph{On {H}erstein's theorems relating modularity in {$A$} and
		{$A^{(+)}$}}, J. Algebra \textbf{163} (1994), no.~1, 207--218.
	
	\bibitem{BCS_22}
	N.~Boudi, Y.~Cabrera~Casado, and M.~Siles~Molina, \emph{Natural families in
		evolution algebras}, Publ. Mat. \textbf{66} (2022), no.~1, 159--181.
	
	\bibitem{thesis_yolanda}
	Y.~Cabrera~Casado, \emph{Evolution algebras}, PhD thesis, Universidad de
	M\'alaga, 2016.
	
	\bibitem{CKM_19}
	Y.~Cabrera~Casado, M.~Kanuni, and M.~Siles~Molina, \emph{Basic ideals in
		evolution algebras}, Linear Algebra Appl. \textbf{570} (2019), 148--180.
	
	\bibitem{CMMT_25}
	Y.~Cabrera~Casado, D.~Mart\'in~Barquero, C.~Mart\'in~Gonz\'alez, and A.~Tocino,
	\emph{Connecting ideals in evolution algebras with hereditary subsets of its
		associated graph}, Collect. Math. \textbf{76} (2025), no.~2, 357--372.
	
	\bibitem{CKO_19}
	L.~M. Camacho, A.~K. Khudoyberdiyev, and B.~A. Omirov, \emph{On the property of
		subalgebras of evolution algebras}, Algebr. Represent. Theory \textbf{22}
	(2019), no.~2, 281--296.
	
	\bibitem{CLOR_14}
	J.~M. Casas, M.~Ladra, B.~A. Omirov, and U.~A. Rozikov, \emph{On evolution
		algebras}, Algebra Colloq. \textbf{21} (2014), no.~2, 331--342.
	
	\bibitem{EL_15}
	A.~Elduque and A.~Labra, \emph{Evolution algebras and graphs}, J. Algebra Appl.
	\textbf{14} (2015), no.~7, 1550103, 10.
	
	\bibitem{GP_25_complete}
	X.~García-Martínez and A.~Pérez-Rodríguez, \emph{A note on complete
		evolution algebras}, 	\href{
		https://doi.org/10.48550/arXiv.2512.12418}{arXiv:2512.12418} (2025).
	
	\bibitem{Ge_76}
	A.~G. Ge{\u{\i}}n, \emph{Semimodular {L}ie algebras}, Sib. Math. J. \textbf{17}
	(1976), no.~2, 189--193.
	
	\bibitem{Gr_98}
	G.~Gr\"atzer, \emph{General lattice theory}, second ed., Birkh\"auser Verlag,
	Basel, 1998.
	
	\bibitem{Ko_65}
	B.~Kolman, \emph{Semi-modular {L}ie algebras}, J. Sci. Hiroshima Univ. Ser. A-I
	Math. \textbf{29} (1965), 149--163.
	
	\bibitem{LPP_25}
	M.~Ladra, P.~P\'aez-Guill\'an, and A.~P\'erez-Rodr\'iguez, \emph{On the
		subalgebra lattice of solvable evolution algebras}, Rev. R. Acad. Cienc.
	Exactas F\'is. Nat. Ser. A Mat. RACSAM \textbf{119} (2025), no.~3, Paper No.
	83, 18 pp.
	
	\bibitem{LP_25_regular}
	M.~Ladra and A.~P\'erez-Rodr\'iguez, \emph{Regular evolution algebras are
		closed under subalgebras}, C. R. Math. Acad. Sci. Paris \textbf{363} (2025),
	1461--1465.
	
	\bibitem{MPS_21}
	N.~Maletesta, P.~P\'aez-Guill\'an, and S.~Siciliano, \emph{Restricted {L}ie
		algebras having a distributive lattice of restricted subalgebras}, Linear
	Multilinear Algebra \textbf{69} (2021), no.~16, 3112--3120.
	
	\bibitem{PST_23}
	P.~P\'aez-Guill\'an, S.~Siciliano, and D.~A. Towers, \emph{On the subalgebra
		lattice of a restricted {L}ie algebra}, Linear Algebra Appl. \textbf{660}
	(2023), 47--65.
	
	\bibitem{ST_22}
	S.~Siciliano and D.~A. Towers, \emph{On the subalgebra lattice of a {L}eibniz
		algebra}, Comm. Algebra \textbf{50} (2022), no.~1, 255--267.
	
	\bibitem{Tian_08}
	J.~P. Tian, \emph{Evolution algebras and their applications}, Lecture Notes in
	Mathematics, vol. 1921, Springer, Berlin, 2008.
	
	\bibitem{TV_06}
	J.~P. Tian and P.~Vojt{\v{e}}chovsk{\'y}, \emph{Mathematical concepts of
		evolution algebras in non-{M}endelian genetics}, Quasigroups Related Systems
	\textbf{14} (2006), no.~1, 111--122.
	
\end{thebibliography}
\end{document}